\documentclass[11pt]{article}
\usepackage{amsmath}
\usepackage{amssymb}
\usepackage{amsthm}
\usepackage{enumerate}
\usepackage[usenames,dvipsnames]{pstricks}
\usepackage{epsfig}
\usepackage{dsfont}
\usepackage{color}
\usepackage{pst-grad}
\usepackage{pst-plot} 
\usepackage[a4paper]{geometry}
\usepackage[utf8]{inputenc}
\usepackage{color}
\newcommand{\N}{\ensuremath{\mathbb{N}}}
\newcommand{\eps}{\ensuremath{\varepsilon}} 
\newcommand{\Op}{\ensuremath{\mathcal{L}}}
\newcommand{\Hyp}{\ensuremath{\mathbb{H}}}
\newcommand{\R}{\ensuremath{\mathbb{R}}}
\newcommand{\Prob}{\ensuremath{\mathbb{P}}}
\newcommand{\ind}{\ensuremath{\mathbf{1}}}
\newcommand{\E}{\ensuremath{\mathbb{E}}}

\newtheorem{lemma}{Lemma}
\newtheorem{theo}{Theorem}
\newtheorem*{theoA}{Theorem A}
\newtheorem*{theoB}{Theorem B}
\newtheorem{prop}{Proposition}
\newtheorem{hypo}{Hypothesis}
\newtheorem{remark}{Remark}
\newtheorem{defi}{Definition}

\author{J\"urgen Angst and Camille Tardif}

\begin{document}
\title{D\'evissage of a Poisson boundary under \\ equivariance and regularity conditions}

\maketitle
\abstract{We present a method that allows, under suitable equivariance and regularity conditions, to determine the Poisson boundary of a diffusion starting from the Poisson boundary of a sub-diffusion of the original one. We then give two examples of application of this d\'evissage method. Namely, we first recover the classical result that the Poisson boundary of Brownian motion on a rotationally symmetric manifolds is generated by its escape angle, and we then give an ``elementary'' probabilistic proof of the delicate result of \cite{ismael}, i.e. the determination of the Poisson boundary of the relativistic Brownian motion in Minkowski space-time.}

\section{Introduction} \label{sec.intro}
The Poisson boundary of a Markov process is a measure space which reflects precisely its long-time asymptotic behavior. In the same time, it can be seen as a random compactification of the state space since it gives some salient information on its geometry at infinity. Finally, it provides a nice representation of bounded harmonic functions associated to the generator of the process, see for example \cite{harry,babillot} for nice introductions to the topic.\par 
\medskip 
Focusing on the case of Brownian motion on Lie groups or Riemannian manifolds, the Poisson boundary can be computed explicitely in a number of examples: semi-simple groups \cite{raugi}, constant curvature Riemannian manifolds and pinched Cartan-Hadamard manifolds \cite{sullivan} etc. Nevertheless, the explicit determination of the Poisson boundary of the Brownian on a general Riemannian manifold is largely out of reach. Indeed, even in the case of Cartan-Hadamard manifolds, the question of its triviality (which is a priori a much simpler problem than its explicit determination) is equivalent to the Green-Wu conjecture on the existence of bounded harmonic functions, see \cite{marcanton} and the references therein.
\par 
\medskip 
In this paper, we present a so-called d\'evissage method that allows, under equivariance and regularity conditions, to determine the Poisson boundary of a diffusion starting from the Poisson boundary of a sub-diffusion of the original one. Namely, if the state space $E$ can be written as $E=X \times G$ in an appropriate coordinate system, where $X$ is a differentiable manifold and $G$ is a finite dimensional connected Lie group, and with 
standard notations recalled in Sect. \ref{sec.backgrounds} below, we prove the following result:

\begin{theoA}[Theorems \ref{theo.trivial}-\ref{theo.gene} below]\label{theo.A}
Let $(x_t,g_t)_{t \geq 0}$ be a diffusion process with values in $X \times G$, starting from $(x,g) \in X \times G$ and satisfying Hypotheses 1 to 4 of Sect. \ref{sec.hypo} below. In particular the first projection $(x_t)_{t \geq 0}$ is itself a diffusion process with values in $X$ and when $t$ goes to infinity, the second projection $(g_t)_{t \geq 0}$ converges $\mathbb P_{(x,g)}-$almost surely to a random element $g_{\infty}$ of $G$. Then, the invariant sigma fields $\textrm{Inv}((x_t, g_t)_{t \geq 0})$ of the full diffusion coincides up to $\mathbb P_{(x,g)}-$negligeable sets with the sum $\textrm{Inv}((x_t)_{t \geq 0}) \vee \sigma(g_{\infty})$. 
\end{theoA}

Under some natural extra hypothesis, the above theorem can be extended to the case where the group $G$ is replaced by a finite dimensional co-compact homogeneous space $Y:=G/K$. 

\begin{theoB}[Theorem \ref{theo.homo} below]\label{theo.B}
Let $(x_t,y_t)_{t \geq 0}$ be a diffusion process with values in $X \times Y$, starting from $(x,y) \in X \times Y$ and satisfying Hypothesis 5 of Sect. \ref{sec.hypo} below. In particular the first projection $(x_t)_{t \geq 0}$ is itself a diffusion process with values in $X$ and when $t$ goes to infinity, the second projection $(y_t)_{t \geq 0}$ converges $\mathbb P_{(x,y)}-$almost surely to a random element $y_{\infty}$ of $Y$. Then, the two sigma fields 
$\textrm{Inv}((x_t, y_t)_{t \geq 0})$ and $\textrm{Inv}((x_t)_{t \geq 0}) \vee \sigma(y_{\infty})$ 
coincide up to $\mathbb P_{(x,y)}-$negligeable sets. 
\end{theoB}
\par 
\medskip 
The plan of the paper is the following: in the next Sect. \ref{sec.devissage}, we specify the geometric and probabilistic backgrounds and then the equivariance and regularity conditions under which the d\'evissage method can be applied. Sect. \ref{sec.proof} is devoted to the proofs of the results stated above: we first consider the case where $\textrm{Inv}((x_t)_{t \geq 0})$ is trivial and $G$ is a finite dimensional Lie group (Theorem \ref{theo.trivial} of Sect. \ref{sec.trivial}), then the case where $\textrm{Inv}((x_t)_{t \geq 0})$ is non-trivial but $G$ is still a finite dimensional Lie group (Theorem \ref{theo.gene} of Sect. \ref{sec.gene}). Finally, we extend this result to case where $Y=G/K$ is a finite dimensional co-compact homogeneous space (Theorem \ref{theo.homo} of Sect. \ref{sec.homo}). 
In Sect. \ref{sec.rot}, we first apply the d\'evissage method to recover the classical result that the Poisson boundary of the standard Brownian motion on a rotationally symmetric manifold is generated by its limit escape angle, see e.g. \cite{march,marcanton2}. To conclude, in Sect. \ref{sec.minkowski}, we give an ``elementary'' probabilistic proof of the main result of \cite{ismael} i.e. the determination of the Poisson boundary of the relativistic Brownian motion in Minkowski space-time.

\section{The d\'evissage method framework} \label{sec.devissage}
\subsection{Geometric and probabilistic background}\label{sec.backgrounds}

Let $X$ be a differentiable manifold and $G$ a finite dimensional connected Lie group, in particular $G$ carries a right invariant Haar measure $\mu$. As usual, let us denote by $C^{\infty}(X \times G, \mathbb R)$ the set of smooth functions from $X \times G$ to the real line $\mathbb R$. From the natural left action of $G$ on itself
\[
\begin{array}{cll}  G \times G & \to & G \\   (g, h) & \mapsto & g.h := gh \end{array},
\]
we deduce a left action of $G$ on $C^{\infty}(X \times G, \mathbb R)$, namely:
\[ 
\begin{array}{cll}  G \times C^{\infty}(X \times G, \mathbb R) & \to & C^{\infty}(X \times G, \mathbb R) \\   (g, f) & \mapsto & g \cdot f := \left( (x,h) \mapsto f(x,g.h) \right)\end{array}.
\]
In this context, let $(x_t, g_t)_{t \geq 0}$ be a diffusion process with values in $X \times G$ and with infinite lifetime. We denote by $\mathcal L$ its infinitesimal generator acting on $C^{\infty}(X \times G, \mathbb R)$. Without loss of generality, we can suppose that the process $(x_t, g_t)_{t \geq 0}$ is defined on the canonical space $(\Omega, \mathcal F)$ where $\Omega:= C(\mathbb R^+,X \times G)$ is the paths space and $\mathcal F$ is its standard Borel sigma field. A generic element $\omega = (\omega_t)_{t \geq 0} \in \Omega$ can then be written $\omega = (\omega^X, \omega^G)$ where $\omega^X =(\omega_t^X)_{t \geq 0} \in C(\mathbb R^+, X)$ and $\omega^G = (\omega_t^G)_{t \geq 0} \in C(\mathbb R^+,G)$. The law of a sample path $(x_t, g_t)_{t \geq 0}$ starting from $(x,g)$ will be denoted by $\mathbb P_{(x,g)}$ and $\mathbb E_{(x,g)}$ will denote the associated expectation. Note that we have again a natural left action of $G$ on $\Omega$:
\[ 
\begin{array}{cll}  G \times \Omega & \to & \Omega\\  (g,\omega) & \mapsto & g.\omega:=(\omega^X, g.\omega^G) \end{array},
\]
where $g.\omega^G := (g.\omega_t^G)_{t \geq 0} \in C(\mathbb R^+, G)$. Without loss of generality, we can also suppose that $(x_t, g_t)_{t \geq 0}$ is the coordinate process, namely: $ x_t(\omega) = \omega_t^X$, $g_t(\omega) = \omega^G_t$, $\forall t \geq 0$. With standard notations, we introduce tail sigma field associated to the diffusion:
$ \mathcal F^{\infty} := \bigcap_{t \geq 0} \sigma( (x_s,g_s), \, s \geq t),$
and we consider the classical shift operators $(\theta_s)_{s \geq 0}$ on $\Omega$:
\[ 
\begin{array}{lcll} \theta_s : &  \Omega & \to & \Omega \\ & \omega=(\omega_t)_{t \geq 0} & \mapsto & \theta_s \omega := (\omega_{t+s})_{t \geq 0} \end{array}.
\]
Recall that, by definition, the invariant sigma field $\textrm{Inv}((x_t,g_t)_{t \geq 0})$ associated to the diffusion process $(x_t, g_t)_{t \geq 0}$ is the sub-sigma field of $\mathcal F^{\infty}$ composed of invariant events, that is events $A \in \mathcal F^{\infty}$ such that $\theta_s^{-1} A=A$ for all $s >0$.
\par
\medskip
If $K$ is a compact subgroup of $G$, we will denote by $Y$ the associated homogeneous space i.e. $Y:= G/K$ and by $\pi$ the canonical projection $\pi : G \to G/K$. As above, given a  diffusion process $(x_t, y_t)_{t \geq 0}$ on $X \times Y$, we will denote by $\mathbb P_{(x,y)}$ the law of the path starting from $(x,y) \in X \times Y$, which we realize as probability measure on the canonical space $\pi(\Omega)= C(\mathbb R^+, X \times Y)$.

\subsection{D\'evissage, convergence, equivariance and regularity conditions} \label{sec.hypo}

In the case of a group i.e. given a diffusion process  $(x_t, g_t)_{t \geq 0}$ with values in $X \times G$, the d\'evissage method can be applied under the following set of hypotheses:

\begin{hypo}[D\'evissage condition]\label{hypo.devissage}The first projection $(x_t)_{t \geq 0}$ is a sub-diffusion of the full process $(x_t,g_t)_{t \geq 0}$. Its own invariant sigma field $\textrm{Inv}((x_t)_{t \geq 0})$ is either trivial or generated by a random variable $\ell_{\infty}$ with values in a separable measure space $(S,\mathcal G, \lambda)$ and the law of $\ell_{\infty}$ is absolutely continuous with respect to $\lambda$. 
 \end{hypo}
 
 \begin{hypo}[Convergence condition]\label{hypo.conv}
For any starting point $(x,g) \in X \times G$, the process $(g_t)_{t \geq 0}$ converges $\mathbb P_{(x,g)}-$almost surely when $t$ goes to infinity to a random variable $g_{\infty}$ in $G$.
\end{hypo}

\begin{hypo}[Equivariance condition]\label{hypo.cov}The infinitesimal generator $\mathcal L$ of $(x_t,g_t)_{t \geq 0}$ is equivariant under the action of $G$ on $C^{\infty}(X \times G, \mathbb R)$, i.e. $\forall f \in C^{\infty}(X \times G, \mathbb R)$, we have
\[ 
\mathcal L(g \cdot f) = g \cdot (\mathcal L f).
\]
\end{hypo}

\begin{hypo}[Regularity condition]\label{hypo.reg} $\mathcal L-$harmonic functions are continuous on $X \times G$.
\end{hypo}

In the homogeneous case, i.e. given a diffusion $(x_t, y_t)_{t \geq 0}$ on $X \times Y$ where $Y=G/K$ is a co-compact homogeneous space, our hypothesis can be formulated as follows:

\begin{hypo}[Homogeneous case]\label{hypo.homo}
There exists a $K$-right equivariant diffusion $(x_t, g_t)_{t \geq 0}$ in $X \times G$ satisfying Hypotheses 1 to 4 above such that under $\mathbb P_{(x,y)}$ the process  $(x_t, y_t)_{t \geq 0}$ has the same law as  $(x_t, \pi(g_t))_{t \geq 0}$ under $\mathbb P_{(x,g)}$ for $g \in \pi^{-1}(\{ y \})$. 
\end{hypo}

\subsection{Comments on the assumptions}\label{sec.comments}

Let us first remark that Hypotheses \ref{hypo.devissage} and \ref{hypo.conv} ensure that the two sigma fields $\textrm{Inv}((x_t)_{t \geq 0}) $ and $\sigma(g_{\infty})$ appearing in Theorem A are well defined. 

\subsubsection{On the d\'evissage condition} \label{sec.devis}

The starting point of the d\'evissage method is that the state space $E$ of the original diffusion can be written as $X \times G$ (resp. $X \times G/K$) in an appropriate coordinate system, where the corresponding first projection $(x_t)_{t \geq 0}$ is a sub-diffusion of $(x_t,g_t)_{t \geq 0}$ (resp. $(x_t,y_t)_{t \geq 0}$). This ``splitting property'' occurs in a large number of situations, in particular when considering diffusion processes on manifolds that show some symmetries. 

For example, any left invariant  diffusion $(z_t)_{t \geq 0}$ with values in a semi-simple Lie group $H$ can be decomposed in Iwasawa coordinates as $z_t= n_t a_t k_t$ where $n_t \in N$, $a_t \in A$, $k_t\in K$ take values in Lie subgroups and $(k_t)_{t \geq 0}$ and $(a_t, k_t)_{t \geq 0}$ are sub-diffusions. In other words, the state space can be decomposed as the product of $X=A \times K$ and $G=N$. Under some regularity conditions (see e.g. \cite{liao}), it can be shown that the Poisson boundary of the sub-diffusion $(a_t, k_t)$ is trivial and that $n_t$ converges almost-surely to a random variable $n_\infty \in N$ when $t$ goes to infinity. Thus, our results ensure that the Poisson boundary of the full diffusion $(z_t)_{t \geq 0}$ is generated by the single random variable $n_\infty$.

Another typical situation where the d\'evissage condition is fulfilled is the case of standard Brownian motion on a Riemannian manifold with a warped product structure, a very representative example being the classical hyperbolic space $\mathbb H^d$ seen in polar coordinates $(r,\theta) \in \mathbb R_+^* \times \mathbb S^{d-1}$, i.e. $X =\mathbb R_+^* $ and $G/K = SO(d)/SO(d-1)$. In that case, the radial component $(r_t)_{t \geq 0}$ is a one-dimensional transient sub-diffusion whose Poisson boundary is trivial and the angular component $(\theta_t)_{t \geq 0}$  is a time-changed spherical Brownian motion on $\mathbb S^{d-1}$ that converges almost surely to a random variable $\theta_{\infty} \in \mathbb S^{d-1}$. Again, the d\'evissage method ensures that  the Poisson boundary of the full diffusion is generated by the single random variable $\theta_\infty$. This example generalizes to the case of a standard Brownian motion on a rotationally symmetric manifold, see Sect. \ref{sec.rot}.
\par
\medskip
Otherwise, the hypothesis that the first projection $(x_t)_{t \geq 0}$ is a sub-diffusion of the full diffusion $(x_t,g_t)_{t \geq 0}$ (resp. $(x_t,y_t)_{t \geq 0}$) is convenient and easy to check when considering examples. Nevertheless it is not necessary and the proof of Theorem A actually applies verbatim in certain cases where  the d\'evissage condition is not fulfilled. For example, when $\textrm{Inv}((x_t)_{t \geq 0}) $ is trivial,  what is really needed is the fact that harmonic functions that are a priori constant in the second variable $g \in G$ are in fact constant in both variables $(x,g) \in X \times G$, see Remark \ref{rem.devissage} at the end of the proof of Theorem \ref{theo.trivial} in Sect. \ref{sec.trivial}.
\par
\medskip
Finally, remark that the absolute continuity condition required when $\textrm{Inv}((x_t)_{t \geq 0}) $ is non-trivial, is ensured for example if the infinitesimal generator of the diffusion process $(x_t)_{t \geq 0}$ is hypoelliptic. Moreover, without loss of generality, we can suppose in that case that the measure $\lambda$ on  $(S,\mathcal G)$ is a probability measure, see \cite{kaima}.

\subsubsection{On the equivariance condition}
The main hypothesis that allows to implement the d\'evissage scheme is Hypothesis \ref{hypo.cov} i.e. the equivariance condition. To emphasize its role, let us first consider the following example where the diffusion process $(x_t,g_t)_{t \geq 0}$ with values $X \times G = \mathbb R \times \mathbb R$ is solution of the stochastic differential equations system: 
\begin{equation}
\left \lbrace \begin{array}{l}
\displaystyle{d x_t = dt + e^{-x^2_t} dB_t}, \\ 
\\
 dg_t = e^{-x_t} dt,
\end{array}\right., \quad (x_0, g_0) \in \mathbb R \times \mathbb R, \label{kaima}
\end{equation}
where $(B_t)_{t \geq 0}$ is a standard real Brownian motion. The infinitesimal generator $\mathcal L$ of the diffusion is hypoelliptic, so that Hypothesis \ref{hypo.reg} is fulfilled. Naturally, the process $(x_t)_{t \geq 0}$ is a one dimensional sub-diffusion of $(x_t,g_t)_{t \geq 0}$ and from the Lemma \ref{lem.example} below, Hypothesis \ref{hypo.devissage} is also fulfilled.

\begin{lemma}\label{lem.example}
There exists a real stochastic process $(u_t)_{t \geq 0}$ that converges $\mathbb P_{(x,g)}-$almost surely to a random variable $u_{\infty}$ in $\mathbb R$ when $t$ goes to infinity such that for all $t \geq 0$
\[
x_t = x_0 + t +u_t.
\]
Moreover, the invariant sigma field $\textrm{Inv}((x_t)_{t \geq 0})$ is trivial.
\end{lemma}
\begin{proof}[Proof of the lemma]
For all $t \geq 0$, we have 
\[ 
x_t = x_0 + t + u_t, \;\; \hbox{where} \;\; u_t:= \int_{0}^t e^{-x^2_s} dB_s.
\]
The martingale $u_t$ satisfies $\langle u \rangle_t = \int_0^t  e^{- 2 x^2_s}ds \leq  t$ so that from the law of iterated logarithm, we have almost surely $x_t \geq t/2$ for $t$ sufficiently large. In particular, $\langle u \rangle_{\infty} < +\infty$ almost surely and $u_t$ is convergent. Since $x_t$ goes almost surely to infinity with $t$, standard shift-coupling arguments apply and we deduce that $\textrm{Inv}((x_t)_{t \geq 0})$ is trivial. 
Note however that the tail sigma field of $(x_t)_{t \geq 0}$ i.e. the invariant sigma field of the space-time process $\textrm{Inv}((t, x_t)_{t \geq 0})$ is not trivial. Indeed, $x_0 + u_{\infty} = \lim_{t \to +\infty} (x_t-t)$ is a non-trivial shift invariant random variable.
\end{proof}

From Lemma \ref{lem.example} again, the second projection $g_t= g_0 + \int_0^t e^{-x_s}ds$ converges $\mathbb P_{(x,g)}-$almost surely when $t$ goes to infinity to a random variable $g_{\infty}$ in $\mathbb R$ and Hypothesis \ref{hypo.conv} is satisfied.
Finally, considering the action of $G=(\mathbb R,+)$ on itself by translation, Hypothesis \ref{hypo.cov} is also satisfied since, for $f \in C(\mathbb R \times \mathbb R, \mathbb R)$ and $(x,g,h) \in \mathbb R^3$ we have
\[ 
\mathcal L (h \cdot f)(x,g) = (\partial_x f)(x, g+h)  + \frac{1}{2} e^{-x^2} (\partial_x^2 f)(x,g+h) + e^{-x} (\partial_g f)(x,g+h) = h \cdot (\mathcal L f)(x,g).
\]
Hence, from Theorem A, the two sigma fields 
$\textrm{Inv}((x_t, g_t)_{t \geq 0})$ and $\textrm{Inv}((x_t)_{t \geq 0}) \vee \sigma(g_{\infty})= \sigma(g_{\infty})$ 
coincide up to $\mathbb P_{(x,g)}-$negligeable sets i.e. the d\'evissage scheme applies.\par
\medskip
Let us now consider a very similar process, namely the diffusion process $(x_t,g_t)_{t \geq 0}$ with values $X \times G = \mathbb R \times \mathbb R$ which is solution of the new following stochastic differential equations system: 
\begin{equation} \label{kaima2}
 \left \lbrace \begin{array}{l}
\displaystyle{d x_t = dt + e^{-x^2_t} dB_t}, \\
\\
d g_t = - g_t dt,
\end{array}\right., \quad (x_0, y_0) \in \mathbb R \times \mathbb R,
\end{equation}
where $(B_t)_{t \geq 0}$ is a standard real Brownian motion. With a view to apply the d\'evissage method, the context seems favorable because the infinitesimal generator $\mathcal L$ of the diffusion is hypoelliptic, $(x_t)_{t \geq 0}$ is a one dimensional sub-diffusion of $(x_t,g_t)_{t \geq 0}$, and $g_t=g_0 e^{-t}$ converges (deterministically) to $g_{\infty}=0$ when $t$ goes to infinity. In particular, the sigma field $\textrm{Inv}((x_t)_{t \geq 0}) \vee \sigma(g_{\infty})$ is trivial. Nevertheless, we have the following proposition: 

\begin{prop}
Let $(x,g) \in \mathbb R \times \mathbb R$ with $g \neq 0$, then the two sigma fields
$\textrm{Inv}((x_t, g_t)_{t \geq 0})$ and $\textrm{Inv}((x_t)_{t \geq 0}) \vee \sigma(g_{\infty})$ differ by a $\mathbb P_{(x,g)}-$non-negligeable set.
\end{prop}

\begin{proof}
If $g \neq 0$, the sigma field $\textrm{Inv}((x_t, g_t)_{t \geq 0})$ is not trivial under $\mathbb P_{(x,g)}$ because the process $x_t + \log(|g_t|)$ converges $\mathbb P_{(x,g)}-$almost surely to $x_0 + \log(|g_0|)+u_{\infty}$ which, from the proof of Lemma \ref{lem.example}, is a non-trivial invariant random variable.
\end{proof}

The reason for which the d\'evissage method does not apply here is that Hypothesis \ref{hypo.cov} i.e. the equivariance condition is not fulfilled. Indeed, the generator of the full diffusion writes
\[
 \mathcal L = \partial_x + \frac{1}{2} e^{-x^2} \partial_x^2 - g \partial_g, 
 \]
and in general, for $f \in C(\mathbb R \times \mathbb R, \mathbb R)$ and $(x,g,h) \in \mathbb R^3$ we have
\[ 
\begin{array}{cll}
\displaystyle{\mathcal L (h \cdot f)(x,g)} & = & \displaystyle{(\partial_x f)(x, g+h)  + \frac{1}{2} e^{-x^2} (\partial_x^2 f)(x,g+h) - g (\partial_g f)(x,g+h)} \\
 \neq &  & \\
\displaystyle{h \cdot (\mathcal L f)(x,g)} & = &  \displaystyle{(\partial_x f)(x, g+h)  + \frac{1}{2} e^{-x^2} (\partial_x^2 f)(x,g+h) - (g+h) (\partial_g f)(x,g+h)}.
\end{array} 
\]

\subsubsection{On the regularity condition} 

As already noticed in the examples of the last section, the regularity condition is automatically satisfied for a large class of diffusion processes, namely when the infinitesimal generator $\mathcal L$ is elliptic or hypoelliptic. The role of this assumption will be clear at the end of the proof of Theorems \ref{theo.trivial} and \ref{theo.gene}, since it allows to go to the limit in the regularization procedure. In a more heuristical way, 
the regularity condition can be seen as a mixing hypothesis which prevents pathologies that may occur when considering foliated dynamics.
 
To be more precise on the kind of pathologies we have in mind, consider the following deterministic example that was suggested to us by S. Gou\"ezel. The underlying space is the product space $X \times Y= \mathbb S^1 \times \mathbb S^1$ where $\mathbb S^1$ is identified to $\mathbb R / \mathbb Z$. Fix $\alpha \notin \mathbb Q$, and define the transformation $T : X \times Y \to X \times Y$ such that $T(x,y):=(x+\alpha, y)$. Now let $X(x,y):=x$ and  $Y(x,y):=y$ be the first and second projections and for $n \geq 0$ define $X_n:=X \circ T^n$ i.e. $X_n(x,y)=(x+n \alpha, y)$ and $Y_n := Y\circ T^n \equiv Y$.  Now, for $y \in \mathbb S^1$, consider the probability measure 
\[ 
\nu_y:= C \sum_{n \in \mathbb Z} \frac{1}{1+n^2} \delta_{ y  + n\alpha},
\]
where $C$ is a normalizing constant and define a measure $\mathbb P$ on $X \times Y$ such that 
\[
\int_{X \times Y} h(x,y) \mathbb P(dx,dy):= \int_{y \in \mathbb S^1} \left[\int_{x \in \mathbb S^1} h(x,y) \left( \frac{1}{2}  \nu_y(dx) +
\frac{1}{2}\nu_{y+1/2}(dx) \right) \right] dy.
\]
Note that the first marginal $\mathbb P_X(\cdot)=\int_Y \mathbb P(\cdot,dy)$ of $\mathbb P$ is the Lebesgue measure hence the invariant sigma field $\textrm{Inv}((X_n)_{n \geq 0})$ is trivial under $\mathbb P$. Since $Y$ is $T-$invariant, the invariant sigma field $\textrm{Inv}((Y_n)_{n \geq 0})$ is composed of events that do not depend on the first coordinate $x$. Nevertheless, the global invariant sigma field $\textrm{Inv}((X_n,Y_n)_{n \geq 0})$ differs from $\textrm{Inv}((Y_n)_{n \geq 0})$ by a $\mathbb P$-non negligeable event 
and the d\'evissage scheme does not apply here. Indeed, both sets 
$A:=\{ (y + n \alpha, y), y \in \mathbb S^1, n \in \mathbb Z\}$ and $B:=\{ (y +1/2 + n \alpha, y), y \in \mathbb S^1, n \in \mathbb Z\}$ are invariant and do depend on the first coordinate.

\section{Proof of the main result}\label{sec.proof}
 
\subsection{Starting from a trivial Poisson boundary}\label{sec.trivial}

We first consider the simplest case when the invariant sigma field of $(x_t)_{t \geq 0}$ is trivial and when $Y=G$ is a finite dimensional Lie group. Namely we will first prove the following result:

\begin{theo}\label{theo.trivial}
Suppose that the full diffusion $(x_t,g_t)_{t \geq 0}$ satisfies Hypotheses 1 to 4. Suppose moreover that for all $(x,g) \in X \times G$, the invariant sigma field $\textrm{Inv}((x_t)_{t \geq 0})$ is trivial for the measure $\mathbb P_{(x,g)}$. Then the two sigma fields  
$$\textrm{Inv}((x_t, g_t)_{t \geq 0}) \;\; \textrm{and}\;\; \sigma(g_{\infty})$$ 
coincide up to $\mathbb P_{(x,g)}-$negligeable sets. Equivalently, if $H$ is a bounded $\mathcal L-$harmonic function, there exits a bounded mesurable function $\psi$ on $G$ such that $H(x,g)=\mathbb E_{(x,g)}[\psi(g_{\infty})]$.
\end{theo}

\begin{proof}[Proof of Theorem \ref{theo.trivial}]
The first step of the proof is the following lemma, which is valid under Hypotheses 1 to 4 (the triviality of $\textrm{Inv}((x_t)_{t \geq 0})$ is not required here). From Hypothesis \ref{hypo.conv}, for all $(x,g) \in X \times G$, the process $(g_t)_{t \geq 0} $ converges $\mathbb P_{(x,g)}-$almost surely to a random variable $g_{\infty}=g_{\infty}(\omega)$ in $G$. 
\begin{lemma}\label{key.lemma}
Under Hypotheses 1 to 4, and for all starting points $(x,g) \in X \times G$ and $h \in G$, the law of the process $h.(x_t,g_t)_{t \geq 0 }=(x_t,h.g_t)_{t \geq 0 }$ under $\mathbb P_{(x, g)}$ coincides with the law of $(x_t,g_t)_{t \geq 0 }$ under $\mathbb P_{(x,h.g)}$. In particular,
\begin{enumerate}
\item the law of the limit $g_{\infty}$ under $\mathbb P_{(x,h.g)}$ is the law of $h. g_{\infty}$ under $\mathbb P_{(x, g)}$;
\item the push-forward measures of both measures $\mathbb P_{(x, g)}$ and $\mathbb P_{(x,h.g)}$ under the mesurable map 
$\omega=(\omega^X, \omega^G) \mapsto h g_{\infty}^{-1}.\omega=(\omega^X, h g_{\infty}^{-1}(\omega).\omega^G)$ coincide.
\end{enumerate}
\end{lemma}
\begin{proof}[Proof of Lemma \ref{key.lemma}]
The result is an direct consequence of the covariance hypothesis \ref{hypo.cov}. Indeed, if $f \in C^{\infty}(X\times G, \mathbb R)$ is compactly supported, from It\^o's formula, under $\mathbb P_{(x,g)}$ we have for all $h \in G$:
$$\begin{array}{ll}
\displaystyle{f(x_t,h.g_t)} & \displaystyle{=(h \cdot f)(x_t,g_t) =(h \cdot f)(x,g) + \int_0^t \mathcal L \left(  h \cdot f\right)(x_s,g_s)ds+M_t} \\
\\
& \displaystyle{=f(x,h.g) +\int_0^t h \cdot \left( \mathcal L  f\right)(x_s,g_s)ds+M_t}\\
\\
& \displaystyle{=f(x,h.g) +\int_0^t  \left(\mathcal L  f\right)(x_s,h .g_s)ds+M_t,}
\end{array}$$
where $M_t$ is a martingale vanishing at zero. Otherwise, under $\mathbb P_{(x,h.g)}$ we have:
$$f(x_t,g_t)=f(x,h.g) + \int_0^t (\mathcal L f)(x_s,g_s) ds +N_t,$$
where $N_t$ is again a martingale vanishing at zero. In other words, under $\mathbb P_{(x, g)}$ and $\mathbb P_{(x,h.g)}$ respectively, both processes $h.(x_t,g_t)_{t \geq 0 }$ and $(x_t,g_t)_{t \geq 0 }$ solve the same martingale problem, hence their laws coincide. 
\end{proof}

Let us go back to the proof of Theorem \ref{theo.trivial}. From Hypothesis \ref{hypo.conv}, for all $(x,g) \in X \times G$, the process $(g_t)_{t \geq 0} $ converges $\mathbb P_{(x,g)}-$almost surely to a random variable $g_{\infty}=g_{\infty}(\omega)$ in $G$. We define
$$\Omega_0^{(x,g)}:= \{\omega \in \Omega, \, \lim_{t \to +\infty} g_t(\omega) \, \hbox{exists}\},$$
and consider the new variable $\tilde{g}_{\infty}$ such that $\tilde{g}_{\infty}:=g_{\infty}$ on $\Omega_0^{(x,g)}$ and $\tilde{g}_{\infty}:=\textrm{Id}_G$ on $\Omega \backslash\Omega_0^{(x,g)}.$
Let $H$ be a bounded $\mathcal L-$harmonic function. There exists a bounded $\textrm{Inv}((x_t,g_t)_{t \geq 0})-$measurable random variable $Z : \Omega \to \mathbb R$, i.e. $Z$ is $\mathcal F^{\infty}-$measurable and satisfies $Z( \theta_s \omega) = Z(\omega)$ for all $\omega \in \Omega$, such that $\forall \, (x,g) \in X \times G$ :
\[
H(x,g) = \mathbb{E}_{(x,g)} [ Z ]. 
\]
Moreover, $(x, g) \in X \times G$ being fixed, for $\Prob_{(x,g)}-$almost all paths $\omega$, we have: 
$$Z(\omega)= \lim_{t \to +\infty} H(x_t(\omega), g_t(\omega)).$$
\noindent
For $h \in G$, consider the new random variable
$$ Z^h(\omega)  := Z( h.\tilde{g}_{\infty}^{-1}.\omega) =Z( \omega^X, h \tilde{g}_\infty(\omega)^{-1}\omega^G ).$$
The variable $Z^h $ is again $\textrm{Inv}((x_t,g_t)_{t \geq 0})-$measurable. Indeed, since the constant function equal to $h$ and the random variable $Z$ are shift-invariant, we have 
$$ Z ( h.\tilde{g}_{\infty}^{-1}(\theta_s \omega).\theta_s \omega) = Z (\theta_s (h.\tilde{g}_{\infty}^{-1}.\omega))= Z(  h.\tilde{g}_{\infty}^{-1}.\omega).$$
Since $Z^h$ is a bounded $\textrm{Inv}((x_t,g_t)_{t \geq 0})-$measurable variable, the function $(x,g) \mapsto \mathbb E_{(x,g)} [Z^{h}]$ is also a bounded $\mathcal L-$harmonic function. But from the second point of Lemma \ref{key.lemma}, for all starting points $(x,g, g') \in X \times G^2$, we have 
$$\mathbb{E}_{(x,g)} [ Z^{h}  ]= \mathbb{E}_{(x,g')} [ Z^{h} ].$$
In other words, the harmonic function $(x, g) \mapsto \mathbb{E}_{(x,g)} [ Z^{h}  ]$ is constant in $g$ and its restriction to $X$ is $\mathcal L^X-$harmonic. Since $\textrm{Inv}((x_t)_{t \geq 0})$ is supposed to be trivial, we deduce that the function $(x, g) \mapsto \mathbb{E}_{(x,g)} [ Z^{h}  ]$ is constant. In the sequel, we will denote by $\psi(h)$ the value of this constant. Note that $h \mapsto \psi(h)$ is a bounded measurable function since $h \mapsto Z^h$ is.
Let us now introduce an approximate unity $(\rho_n)_{n \geq 0}$ on $G$, fix $\mathbf{g} \in G$, $n \in \mathbb N$ and consider the ``conditionned and regularized'' version $Z$, namely:
$$Z^{ \mathbf{g},n}(\omega):= \int_{G} Z^h(\omega) \rho_n( \mathbf{g}h^{-1})\mu(dh). $$
The exact same reasoning as above shows that $Z^{ \mathbf{g},n}$ is a bounded $\textrm{Inv}((x_t,g_t)_{t \geq 0})-$measurable variable so that the function  $(x, g) \mapsto \mathbb{E}_{(x,g)} [ Z^{\mathbf{g},n}  ] $ is constant. Hence, for all $\mathbf{g} \in G$, $n \in \mathbb N$ and $(x,g) \in X \times G$, there exists a set  $\Omega^{\mathbf{g},n,(x,g)} \subset \Omega$ such that $\Prob_{(x,g)}(\Omega^{\mathbf{g},n,(x,g)} ) =1$ and such that for all paths $\omega$ in $\Omega^{\mathbf{g},n,(x,g)}$, we have:
$$\begin{array}{ll}
Z^{\mathbf{g},n}  (\omega) &= \lim_{t \to \infty}  \mathbb{E}_{(x_t(\omega),g_t(\omega))} [ Z^{\mathbf{g},n}   ] \\
\\
 &= \mathbb{E}_{(x_0(\omega),g_0(\omega))} [ Z^{\mathbf{g},n}   ] \\
 \\
 &= \mathbb{E}_{(x,g)} [Z^{\mathbf{g},n}   ].
\end{array}$$
Let $D$ be a countable dense set in $G$ and consider the intersection
 \[
 \Omega^{(x,g)} := \underset{\mathbf{g} \in D, n \in \N}{\bigcap } \Omega^{\mathbf{g},n, (x,g)}.
 \]
We have naturally $\Prob_{(x,g)} ( \Omega^{(x,g)} )=1$ and for $\omega \in \Omega^{(x,g)}$:
\[
\forall \mathbf{g} \in D, \ n \in \N, \quad Z^{\mathbf{g},n} (\omega) = \mathbb{E}_{(x,g)} [ Z^{\mathbf{g},n} ].
\]
Since the above expressions are continuous in $\mathbf{g}$, we deduce that the last inequality is true for all $\mathbf{g} \in G$. In other words, we have 
shown that for all $\mathbf{g} \in G$ and for all $ \omega $ in $\Omega^{(x,g)}$:
\[
 Z^{\mathbf{g},n} (\omega)  =\mathbb{E}_{(x,g)} [ Z^{\mathbf{g},n} ]= \int_G \psi(h ) \rho_n(\mathbf{g}h^{-1})\mu(dh).
\]
In particular, taking $\mathbf{g}= \tilde{g}_\infty (\omega)$, we obtain that for all $\omega \in \Omega^{(x,g)}$ and for all $n \in \N$:
\begin{equation}\label{eq.chgvar}
\displaystyle{Z^{\tilde{g}_{\infty}(\omega),n} (\omega)}  =  \displaystyle{\int_G \psi(h ) \rho_n(\tilde{g}_{\infty}(\omega)h^{-1})\mu(dh)}.
\end{equation}
Recall that the Haar measure $\mu$ is right invariant so that 
\[
\displaystyle{Z^{\tilde{g}_{\infty}(\omega),n} (\omega)}   = \int_{G} Z^h(\omega) \rho_n( g_{\infty}(\omega) h^{-1})\mu(dh)  = \displaystyle{\int_G Z(h.\omega) ) \rho_n(h^{-1})\mu(dh) }, 
\]
and
\[
 \displaystyle{\int_G \psi(h ) \rho_n(\tilde{g}_{\infty}(\omega)h^{-1})\mu(dh)}= \displaystyle{\int_G \psi(h \tilde{g}_{\infty}(\omega)) \rho_n(h^{-1})\mu(dh)}.
\]
Thus, Equation (\ref{eq.chgvar}) is equivalent to
\[
\displaystyle{\int_G Z(h.\omega)  \rho_n(h^{-1})\mu(dh)} =\displaystyle{\int_G \psi(h \tilde{g}_{\infty}(\omega)) \rho_n(h^{-1})\mu(dh)}.
\]
Taking the integral in $\omega$ with respect to $\Prob_{(x,g)}$ on $\Omega^{(x,g)}$, we deduce that for all $ n \in \N$:
\[
\displaystyle{\int_G \mathbb{E}_{(x,g)} [ Z(h.\omega) ]\rho_n(h^{-1})\mu(dh)} =   \displaystyle{\int_G \mathbb{E}_{(x,g)} [ \psi(h\tilde{g}_{\infty}) ] \rho_n(h^{-1})\mu(dh) },
\]
which, from Lemma \ref{key.lemma} yields
$$\displaystyle{\int_G H(x, h g ) \rho_n(h^{-1})\mu(dh) } =  \displaystyle{\int_G \mathbb{E}_{(x,hg)} [ \psi(\tilde{g}_\infty )  ] \rho_n(h^{-1})\mu(dh).}$$
From Hypothesis \ref{hypo.reg}, bounded $ \mathcal L-$harmonic functions are continuous, hence we can let $n$ go to infinity in the above expressions to get the desired result, namely:
$$H(x,g) = \mathbb{E}_{(x,g)} [\psi(g_\infty)].$$

\end{proof}

\begin{remark}\label{rem.devissage}
As already noticed in Sect. \ref{sec.devis}, the Hypothesis that $(x_t)_{t \geq 0}$ is a sub-diffusion of $(x_t, g_t)_{t \geq 0}$ is not necessary. In the proof above, this assumption is only used to ensure that the function $(x, g) \mapsto \mathbb{E}_{(x,g)} [ Z^{h}  ]$, which is constant in $g$, is in fact also constant in $x$, since its restriction to $X$ is $\mathcal L^X-$harmonic. Hence, if one knows a priori that harmonic functions that are constant in the second variable $g \in G$ are in fact constant in both variables $(x,g) \in X \times G$, the end of the proof applies verbatim. This fact allows to implement the d\'evissage method in some situations, where Hypothesis \ref{hypo.devissage} is not satisfied. Consider for example the diffusion process $(x_t, g_t)_{t \geq 0}$ with values in $\mathbb R \times \mathbb R$ which is solution of the stochastic differential equations system:
\begin{equation}\label{kaima3}
\left \lbrace \begin{array}{l}
\displaystyle{d x_t = (1+ g_t) dt +e^{-x^2_t} dB_t}, \\
\\
d g_t = e^{-x_t} dt,
\end{array}\right., \quad (x_0, g_0) \in \mathbb R \times \mathbb R^+,
\end{equation}
where $(B_t)_{t \geq 0}$ is a standard real Brownian motion. In that case, $(x_t)_{t \geq 0}$ is not a sub-diffusion of $(x_t, g_t)_{t \geq 0}$ so that Hypothesis \ref{hypo.devissage} is not fulfilled. Nevertheless, Hypotheses \ref{hypo.conv}-\ref{hypo.reg} are satisfied and one easily checks that $\mathcal L-$harmonic functions that are constant in $g$ are also constant in $x$, and we can conclude that $\textrm{Inv}((x_t, g_t)_{t \geq 0})$ and $\sigma(g_{\infty})$ coincide up to $\mathbb P_{(x_0,g_0)}-$negligeable sets.
\end{remark}

\subsection{Starting from a non-trivial Poisson boundary}\label{sec.gene}

Let us now consider the general case when $\textrm{Inv}((x_t)_{t \geq 0})$ is not trivial but generated by a random variable $\ell_{\infty}$ with values in a separable measure space $(S, \mathcal G)$. We will prove the following result:
\begin{theo}\label{theo.gene}
Suppose that the full diffusion $(x_t,g_t)_{t \geq 0}$ satisfies Hypotheses 1 to 4. Then, for all starting points $ (x,g) \in X \times G$, the two sigma fields  
$$\textrm{Inv}((x_t, g_t)_{t \geq 0}) \;\; \textrm{and}\;\; \sigma(\ell_{\infty},g_{\infty})$$ 
coincide up to $\mathbb P_{(x,g)}-$negligeable sets. Equivalently, if $H$ is a bounded $\mathcal L-$harmonic function, there exits a bounded mesurable function $\psi$ on $S\times G$ such that $H(x,g)=\mathbb E_{(x,g)}[\psi(\ell_{\infty},g_{\infty})]$.
\end{theo}

\if{\begin{remark}
\textcolor{red}{Une approche possible serait le conditionnement par Doob, et d'appliquer le théorème 1. Pour les besoins de la cause, nous donnons une preuve directe en adaptant la preuve précédente}
\end{remark}
}\fi

\begin{proof}The proof is very similar to the one of Theorem \ref{theo.trivial}, but it requires an extra argument to ensure the mesurability of the function $\psi$. So let $H$ be a bounded $\mathcal L-$harmonic function and $Z : \Omega \to \mathbb R$ the associated bounded $\textrm{Inv}((x_t,g_t)_{t \geq 0})-$measurable random variable.  For $\mathbf g, h \in G$ and $n \in \mathbb N$, we consider the random variables
$$ \begin{array}{ll}   Z^h(\omega) & := Z( h.\tilde{g}_{\infty}^{-1}.\omega), \\
\\
\displaystyle{Z^{ \mathbf{g},n}(\omega)} & := \displaystyle{\int_{G} Z^h(\omega) \rho_n( \mathbf{g}h^{-1})\mu(dh). }\end{array}
$$

As in the proof of Theorem \ref{theo.trivial}, the element $h$ being fixed,  the variable $Z^h$ is bounded and $\textrm{Inv}((x_t,g_t)_{t \geq 0})-$measurable, so that the function $(x,g) \mapsto \mathbb E_{(x,g)} [Z^{h}]$ is bounded and $\mathcal L-$harmonic. From Lemma \ref{key.lemma},  this function is constant in $g$ and its restriction to $X$ is thus $\mathcal L^X-$harmonic. Hence, there exists a bounded mesurable function $\psi_h : S \to \mathbb R$ such that 
\begin{equation}\forall (x,g) \in X\times G, \; \;\mathbb E_{(x,g)} [Z^{h}] = \mathbb E_{(x,g)} [\psi_h(\ell_{\infty})].\label{1}\end{equation}
In the same way, for any $A \in \mathcal G$, we have:
 \begin{equation}\E_{x,g} [ \ind_{\ell_\infty \in A } Z^h ] = \E_{x,g} [ \ind_{\ell_\infty \in A} \phi^h(\ell_\infty) ] = \int \ind_{A}(\ell) \psi_h(\ell) k(x,\ell) \lambda(d\ell).  \label{2} \end{equation}
Let us fix $x_0 \in X$ and define 
 $$\mathbb Q_h(d \ell):= \frac{\psi_h(\ell)}{\mathbb E_{(x_0,g)}[Z^h]} k(x_0, \ell) \lambda(d \ell).$$
For each $h$, $\mathbb Q_h$ is absolutely continuous with respect to $k(x_0, \ell) \lambda(d \ell)$ and, by \eqref{2}, $(\mathbb{Q}_h)_h$ is a mesurable family of probability measures. Since $(S, \mathcal G)$ is separable, Theorem 58 p. 57 of \cite{dellmeyer} applies and there exists a measurable map $X : S \times G \to \mathbb R$ such that $X(.,h)$ is a density of $\mathbb Q_h$ with respect to $k(x_0, \ell) \lambda(d \ell)$, i.e. for all $h \in G$
$$ X(\ell,h)  = \frac{\psi_h(\ell)}{\mathbb E_{(x_0,g)}[Z^h]} \;\; \hbox{for} \;\; \lambda-\hbox{almost all}  \; \ell.$$
The map $h \mapsto \mathbb E_{(x_0,g)} [Z^h]$ being measurable, the function
$$\widetilde{\psi}(\ell,h):= X(\ell ,h) \mathbb E_{(x_0,g)}[Z^h],$$ 
is measurable and for all $(x,g) \in X\times G$:
\begin{equation}
\mathbb E_{(x,g)} [Z^{h}] = \mathbb E_{(x,g)} [\widetilde{\psi}(\ell_{\infty},h)].
\end{equation}
For all $\mathbf{g} \in G$, $n \in \mathbb N$ and $(x,g) \in X \times G$, we thus have:
\begin{align*}
\mathbb{E}_{(x,g)} [Z^{\mathbf{g}, n}] &= \int_G \mathbb{E}_{(x,g)} [ Z^h] \rho_n(\mathbf{g}h^{-1})\mu(dh) \\
&= \int_G\mathbb{E}_{(x,g)}  [ \widetilde{\psi}(\ell_{\infty}, h) ] \rho_n(\mathbf{g}h^{-1})\mu(dh)\\
&= \mathbb{E}_{(x,g)}  \left [  \int_G  \widetilde{\psi}(\ell_{\infty}, h) \rho_n(\mathbf{g}h^{-1})\mu(dh) \right ].
\end{align*}
Hence, $(x,g) \in X \times G$ being fixed, both random variables $Z^{\mathbf{g}, n}$ and $ \int_G \widetilde{\psi}(\ell_{\infty},h) \rho_n(\mathbf{g}h^{-1})\mu(dh)$ coincide on a set $\Omega^{\mathbf{g},n,(x,g)} \subset \Omega$ such that $\mathbb P_{(x,g)} (\Omega^{\mathbf{g},n,(x,g)})=1$. As before, if $D$ is a countable dense set in $G$, for all $\omega \in  \Omega^{(x,g)} := \cap_ {\mathbf{g} \in G, n \in \mathbb N} \Omega^{\mathbf{g},n,(x,g)}$:
\[
\forall \mathbf{g} \in D, n \in\mathbb{N}, \quad Z^{\mathbf{g}, n}(\omega) = \int_G  \widetilde{\psi}( \ell_\infty(\omega),h) \rho_n(\mathbf{g} h^{-1}) \mu(dh).
\]
The above expressions being continuous in  $\mathbf{g}$, we can take  $\mathbf{g}=\tilde{g}_\infty(\omega)$ to get 
\begin{align*}
\forall \omega \in \Omega^{(x,g)}, \forall n \in \mathbb{N}, \quad  Z^{\tilde{g}_\infty(\omega),n}(\omega) &= \int_G \widetilde{\psi} (\ell_\infty(\omega),h) \rho_n(\tilde{g}_\infty(\omega) h^{-1}) \mu(dh)  \\
&= \int_G \widetilde{\psi}(\ell_\infty(\omega), h \tilde{g}_\infty(\omega)) \rho_n(h^{-1}) \mu(dh).
\end{align*}
Taking the expectation under $\Prob_{(x,g)}$, the left hand side gives : 

\[
\begin{array}{ll}
\displaystyle{\mathbb{E}_{(x,g)} [Z^{\tilde{g}_{\infty},n}] } & = \displaystyle{\int_{\Omega} \left[   \int_G Z^{h}(\omega) \rho_n(\tilde{g}_{\infty} (\omega)h^{-1}) \mu(dh)    \right] \mathbb P_{(x,g)}(d\omega)}\\
\\
 & = \displaystyle{ \int_{\Omega}\left[   \int_G Z^{h \tilde{g}_{\infty}(\omega)} (\omega)  \rho_n( h^{-1}) \mu(dh)    \right] \mathbb P_{(x,g)}(d\omega)}\\
 \\
  & =  \displaystyle{\int_{\Omega}\left[   \int_G Z ( h.\omega)  \rho_n( h^{-1}) \mu(dh)    \right] \mathbb P_{(x,g)}(d\omega)}\\
  \\
    & =\displaystyle{  \int_{\Omega}\left[   \int_G Z ( \omega)  \rho_n(h^{-1}) \mu(dh)    \right] \mathbb P_{(x,h g)}(d\omega)}\\
  \\
& = \displaystyle{\int_G \mathbb{E}_{(x,h g)} [Z] \rho_n(h^{-1})\mu(dh)}
\end{array}
\]
and the right hand side
\[
\begin{array}{ll}
 \displaystyle{ \mathbb{E}_{(x,g)} \left[ \int_G \widetilde{\psi} (\ell_{\infty}, h \tilde{g}_\infty)  \rho_n(h^{-1})\mu(dh) \right]} & = \displaystyle{\int_G \mathbb{E}_{(x,g)} \left[ \widetilde{\psi} (\ell_{\infty}, h \tilde{g}_\infty) \right] \rho_n(h^{-1})\mu(dh)}\\
\\
&  = \displaystyle{\int_G \mathbb{E}_{(x,hg)} \left[ \widetilde{\psi} (\ell_{\infty}, \tilde{g}_\infty) \right] \rho_n(h^{-1})\mu(dh)}
\end{array}
\]
Since $\mathcal L-$harmonic functions are continuous, letting $n$ go to infinity, we deduce
\[
\mathbb{E}_{(x,g)} [Z ] = \mathbb{E}_{(x,g)} [\widetilde{\psi}(\ell_{\infty}, \tilde{g}_\infty)].
\]
\end{proof}

\subsection{Extension to homogeneous manifolds}\label{sec.homo}

Consider $K$ a compact sub-group of $G$, denote by $Y:= G/K$ the homogenous space associated and $\pi : G \to G/K$ the canonical projection. Let $(x_t, y_t)$ be  a diffusion on $X \times Y$. 

\begin{theo}\label{theo.homo}
Suppose that the full diffusion $(x_t,y_t)_{t \geq 0}$ satisfies Hypothesis 5 of Sect. \ref{sec.hypo}, then for all starting points $(x,y) \in X \times Y$, the two sigma fields  
$$\textrm{Inv}((x_t, y_t)_{t \geq 0}) \;\; \textrm{and}\;\; \textrm{Inv}((x_t)_{t \geq 0}) \vee \sigma(y_{\infty})$$ 
coincide up to $\Prob_{x,y}$-negligeable sets. 
\end{theo}

\begin{proof}
 We consider here the case where $\textrm{Inv}((x_t)_{t \geq 0})$ is generated by a random variable $\ell_{\infty}$ with values in a separable measure space $(S,\mathcal G, \lambda)$. The case where $\textrm{Inv}((x_t)_{t \geq 0})$ is trivial can be treated is a very similar way. 
 Let us fix $(x,y) \in X \times Y$ and $g \in \pi^{-1}(\{ y \})$.
By Hypothesis 5, there exists a $K$-right equivariant diffusion $(x_t, g_t)_{t \geq 0}$ on $X \times G$ such that,  under $\Prob_{(x,g)}$, the process $(x_t, \pi ( g_t))_{t \geq 0}$  has the same law as the process $(x_t, y_t)_{t \geq 0}$ under $\Prob_{(x, \pi(g))}$. Moreover, $(x_t, g_t)_{t \geq 0}$ satisfies Hypotheses \ref{hypo.devissage}-\ref{hypo.reg}, in particular $g_t$ converges $\Prob_{(x,g)}-$almost surely to $g_{\infty}$ when $t$ goes to infinity. Hence, $y_t$ converges $\Prob_{(x, \pi(g))}-$almost surely to the asymptotic random variable 
\[
y_{\infty} : \omega \in \pi(\Omega)  \mapsto \left \{ \begin{array}{ll} \displaystyle{ \mathrm{lim}_{t \to + \infty} y_t (\omega ) \ \mathrm{if \ exists}, }  \\ \pi(\mathrm{Id}) \ \mathrm{else}. \end{array} \right .
\]  
Moreover, for any $g$ in $\pi^{-1}(\{ y \})$, the law of $y_\infty$ under $\Prob_{(x,y)}$ is the same as the law of $\pi(g_{\infty})$ under $\Prob_{(x,g)}$.
So let us consider $Z : \pi(\Omega) \to \R$ a bounded $\theta_t$-invariant random variable, for any $g$ in $\pi^{-1}(\{ y \})$, we have:
\[
\mathbb{E}_{(x,y)} [ Z ] = \mathbb{E}_{(x,g)} [ Z \circ \pi] = \int_K \mathbb{E}_{(x,gk)} [ Z \circ \pi] \mathrm{Haar}(dk).
\]
Since the variable $Z \circ \pi$ is $\textrm{Inv}((x_t, g_t)_{t \geq 0})-$mesurable and bounded, we obtain by Theorem \ref{theo.gene} applied to $(x_t, g_t)_{t \geq 0}$ that there exists a bounded mesurable function $(\ell, g) \mapsto \widetilde{H}(\ell, g)$ such that:
\[
\forall (x,g) \in X \times G, \;  \forall k \in K, \quad \mathbb{E}_{(x,gk)} [ Z \circ \pi] = \mathbb{E}_{(x,gk)} [ \widetilde{H}(\ell_\infty, g_\infty )].
\]
Then, using the $K$-right equivariance of $(x_t, g_t)_{t \geq 0}$ we obtain for $g \in \pi^{-1} ( \{ y \} )$:
\begin{align*}
\mathbb{E}_{(x,y)} [ Z ] &= \int_K \mathbb{E}_{(x,gk)} [ \widetilde{H}(\ell_\infty, g_\infty )] \mathrm{Haar}(dk) \\
&=  \int_K \mathbb{E}_{(x,g)} [ \widetilde{H}(\ell_\infty, g_\infty k )] \mathrm{Haar}(dk).
\end{align*}
Now introduce $\mathcal{S}: Y \to G$ a measurable section of $\pi$. Then $g_\infty =  \mathcal{S}(\pi(g_\infty) ) k'$ for some random $k' \in K$ and we have
\begin{align*}
\mathbb{E}_{(x,y)} [ Z ] &= \mathbb{E}_{(x,g)} \left [  \int_K  \widetilde{H}(\ell_\infty, \mathcal{S}(\pi(g_\infty) ) k' k ) \mathrm{Haar}(dk)    \right ] \\
&= \mathbb{E}_{(x,g)} \left [  \int_K  \widetilde{H}(\ell_\infty, \mathcal{S}(\pi(g_\infty) ) k ) \mathrm{Haar}(dk) \right ].
\end{align*}
Finally denoting by $H$ the bounded measurable function on $S \times Y$ defined by  
\[
H(\ell, y) := \int_K \widetilde{H}( \ell, \mathcal{S} (y) k ) \mathrm{Haar}(dk),
\] 
we have 
\[ 
\mathbb{E}_{(x,y)} [ Z ]  = \mathbb{E}_{(x,y)} [ H(\ell_\infty, y_\infty) ].
\]

\end{proof}

\section{Examples of application}\label{sec.appli}
In this last section, we give two examples of application of the d\'evissage method. The first one, which concerns the asymptotic behavior of the standard Brownian motion on a rotationally symmetric manifold, is a direct consequence of Theorem \ref{theo.homo}. The second one, which caracterizes the Poisson boundary of the relativistic Brownian motion in Minkowski space is an application of Theorem \ref{theo.trivial}, after a suitable change of coordinates.

\subsection{Brownian motion on rotationally invariant models}\label{sec.rot}

Let $M= \mathbb{R}^{*}_{+} \times \mathbb{S}^{n-1} \cup \{ o \}$ be a rotationally invariant model, diffeomorphic to $\mathbb{R}^{n}$, with center $o \in M$ and metric
\[g = dr^2 + f^2(r) d\theta^2.\]
The classical Laplacian $\Delta^M$ on $M$ is given by
\[ 
\Delta^M = \partial^2_r  + (n-1)\frac{f'}{f}(r)\partial_r + \frac{1}{f(r)^2}\Delta^{\mathbb{S}^{n-1}}_{\theta},
\]
where $\Delta^{\mathbb{S}^{n-1}}_{\theta}$ is the classical Laplacian on the round sphere $\mathbb{S}^{n-1}$.

Let $X$ be Brownian motion on $(M, g)$ starting from $X_0 = x_0 \neq o$ which is decomposed
according to $M\backslash \{o\}= ]0, + \infty[ \times \mathbb{S}^{n-1}$ into its radial and angular process, namely $X_t = (r_t,\theta_t)$. It solves the following system of stochastic differential equations:
\[ 
\left \lbrace \begin{array}{ll} 
\displaystyle{dr_t} &= \displaystyle{d W_t + \frac{n-1}{2} \frac{f'}{f}(r_t) dt}, \\
\displaystyle{d \theta_t} &= \displaystyle{\frac{1}{f(r_t)^2} d \Theta_t},
\end{array}
\right.
\]
where $W_t$ and $\Theta_t$ are independent Brownian motion on $\mathbb{R}$ and  $\mathbb{S}^{n-1}$ respectively. Notice that the radial component $r_t$ is a one dimensional sub-diffusion of $X_t$ and that the general theory of one-dimensional diffusion ensures that (see for example \cite{marcanton2})  
\begin{enumerate} \label{intcond}
\item if \[ \int_{1}^{+\infty} f^{1-n}(r)dr < +\infty, \] then $r_t$ goes to infinity almost surely;
\item if \[\int_{1}^{+\infty} f^{n-1}(r) \left(\int_{r}^{+\infty} f^{1-n}(\rho)d\rho\right) dr = +\infty, \] then the lifetime of $r_t$ is $+\infty$;
\item if \[  \int_{1}^{+\infty} f^{n-3}(r) \left(\int_{r}^{+\infty} f^{1-n}(\rho)d\rho\right) dr < +\infty, \] then almost surely  \[ \int_0^{+ \infty} \frac{1}{f(r_s)^2} ds < + \infty.\]
\end{enumerate}

Thus, under these integrability conditions on $f$, the Brownian motion $X_t$ in $M$ is transient, does not explode and its angular part $\theta_t$ converges almost-surely to some asymptotic random variable $\theta_\infty \in \mathbb S^{n-1}$.
In other words, $X_t$ goes to infinity in a random preferred direction or equivalently, the model $M$ being seen as the interior of the unit ball $\mathbb B$, its converges to a random point $\theta_{\infty}$ of its boundary.

\begin{figure}[ht]
\begin{center}
\includegraphics[scale=0.7]{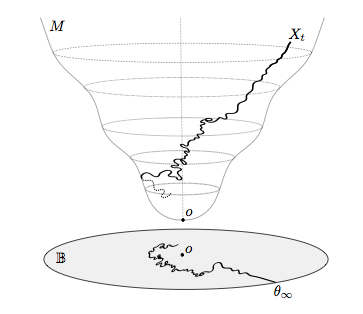}
\end{center}
\caption{Asymptotic behavior of $X_t$ under conditions 1-3.\label{fig.rot}}
\end{figure}

Thus, choosing the natural spherical coordinates on $M$, and considering the sphere $\mathbb{S}^{n-1}$ as a $SO(n)$-homogeneous space, we are in position to apply Theorem \ref{theo.homo} to obtain

\begin{theo}
Under conditions 1-3 on the warping function $f$, the Poisson boundary of the Brownian motion $X_t=(r_t, \theta_t)$ on $M$ is generated by its escape angle $\theta_\infty$. 
\end{theo}

\begin{proof}
The devissage and regularity conditions are fulfilled and since the sub-diffusion $r_t$ goes asymptotically to infinity, we obtain by shift coupling that its invariant $\sigma$-algebra is almost-surely trivial.
It remains to check that Hypothesis 5, concerning the homogenous case, is fulfilled. Let us denote by 
\[ \pi: SO(n) \longrightarrow \mathbb{S}^{n-1}=SO(n)/SO(n-1) \]
the canonical projection which makes $SO(n)$ the orthonormal frame bundle over $\mathbb{S}^{n-1}$.  We can lift horizontally the $\mathbb{S}^{n-1}$-Brownian motion $(\Theta_t)_{t\geq 0}$ into a left invariant diffusion living in $SO(n)$. Namely, denoting by $(H_i)_{i=1\dots n-1}$ the canonical horizontal vector fields on $SO(n)$  (which are moreover left invariant) we have
\[
\forall \phi \in C^2(\mathbb{S}^{n-1}, \mathbb{R}), \quad \Delta(\phi) \circ \pi = \left ( \sum_{i=1}^{n-1} H_i^2 \right ) ( \phi \circ \pi). 
\]
Thus denoting by $(r_t, g_t)$ the diffusion on $]0, + \infty [ \times SO(n)$ generated by 
\[
\Op:= \partial^2_r  + (n-1)\frac{f'}{f}(r)\partial_r + \frac{1}{f(r)^2}  \sum_{i=1}^{n-1} H_i^2,
\]
we obtain that $(r_t, \pi(g_t))$ is a Brownian motion on $M$. Moreover, according that $\int^{+ \infty} f(r_s)^{-2} ds$ is  almost surely finite, the  process $g_t$ converges almost surely to some $SO(n)$-valued asymptotic random variable $g_{\infty}$. Note that, $(H_i)_{i=1\dots n-1}$ being left-invariant, the equivariance condition is  satisfied, and finally  Hypothesis 5 is fulfilled.
\end{proof}

\subsection{Relativistic Brownian motion in Minkowski space}\label{sec.minkowski}

Both Euclidean Brownian motion $B_t$ and  kinematic Ornstein-Uhlenbeck process are standard models for the physical Brownian motion. They are non relativistic models because, in both cases, the reference frame in which the fluid is at rest plays a specific role (taking into account the fluid viscosity). For instance the dynamics of those processes change when a constant drift is added to the frame. So, in both models, there is no Galilean covariance and a fortiori there is no Lorentzian covariance neither. Nevertheless, it is remarkable that when the viscosity coefficient of the fluid is null the kinematic Ornstein-Uhlenbeck process simply writes $(B_t, \int^t B_s ds)$ and it shows a Galilean covariant dynamics. In 1966, Dudley  introduced in \cite{dudley1} a Lorentzian analogue to this process, more precisely he proved that there exists a unique diffusion process, taking  values in the Minkowskian phase space and having a Lorentzian covariant dynamics with time-like $C^1$ trajectories. In \cite{ismael}, Bailleul caracterized the long-time asymptotic behavio of this relativistic diffusion by computing its Poisson boundary. He showed in particular that it corresponds to the causal boundary to Minkowski space-time. We propose in this section to use Theorem \ref{theo.trivial} to provide a direct proof of his result.

\if{Both Euclidean Brownian motion $B_t$ and  Ornstein-Uhlenbeck process are non relativistic models for the physical Brownian motion. They are non relativistic because The reference frame in which the fluid is at rest play a specific role (taking into account the fluid viscosity). For instance the dynamics of those processes change when a constant drift is added to the frame. So there is no Galilean covariance dynamics and a fortiori there is no Lorentzian covariance neither. Nevertheless when the viscosity coefficient of the fluid is null the Ornstein-Uhlenbeck process writes $(B_t, \int^t B_s ds)$ and has Galilean covariance dynamics. In 1966, Dudley \cite{dudley1} introduced a Lorentzian analogue to this process. It is a diffusion taking  values in the Lorentzian phase space and having Lorentzian covariance dynamics with  time-like $C^1$ trajectories in Minkowski space-time.  Bailleul \cite{ismael} found the Poisson boundary of this process and showed that it corresponds to the causal boundary to Minkowski space-time. We propose in this section to use Theorem \ref{theo.gene} to provide a direct proof of his result. 
}
\fi

\paragraph{Dudley diffusion in Minkowski space-time.} We denote by $\R^{1,d}$ the Minkowski space-time $\R^{d+1}$ endowed with the Lorentz quadratic form $q$
\[
q(\xi):= \left ( \xi^0  \right)^2 - \sum_{i=1}^d  \left ( \xi^i  \right)^2.
\]
The canonical basis is denoted by $(e_0, \dots, e_d)$.
Let denote by $\Hyp^d$ the half  pseudo unit sphere 
\[
\Hyp^d:= \{ \xi \in \R^{1,d}, \ q(\xi) =1 \ \xi^{0} > 0 \}.
\]
The restriction of $q$ to $T\Hyp^d$ makes $\Hyp^d$ a Riemmanian manifold of constant negative curvature, so $\Hyp^d$ is the hyperbolo\"id model of the hyperbolic space of dimension $d$. 

%
Via the following polar coordinates
\[
\begin{matrix}
\mathbb{R}^{*}_{+} \times \mathbb{S}^{d-1} & \longrightarrow &  \mathbb{H}^{d} \setminus \{ e_0 \} \\
(r, \theta) & \longmapsto & \displaystyle{\cosh(r)e_0 + \sinh(r) \sum_{i=1}^{d} \theta^i e_i}
\end{matrix}
\]
$\mathbb{H}^d$ is a rotationally invariant model centered at $e_0$ with metric
$
(ds)^2= (dr)^2 + \sinh(r) (d \theta )^2, 
$
and the Laplacian is given by
\[
\Delta_{r, \theta}^{\mathbb{H}^d}= \partial^2_r + (d-1) \coth(r) \partial_r + \frac{1}{\sinh(r)^2} \Delta^{\mathbb{S}^{d-1}}_\theta. 
\]
Let now define Dudley's diffusion introduced in \cite{dudley1}.
\begin{defi}
Dudley's diffusion is the diffusion process $(\dot{\xi}_t, \xi_t)$ with values in $\Hyp^d \times \R^{1,d}$ generated by 
\[
\Op:= \frac{\sigma^2}{2} \Delta^{\Hyp^d}_{\dot{\xi}} + \dot{\xi} \cdot \partial_{\xi},
\]
thus $\dot{\xi}_t$ is a classical Riemannian Brownian motion in $\Hyp^d$ and $\xi_t= \xi_0 + \int_0^t \dot{\xi}_s ds$.
\end{defi}

Note that  the paths $\xi_t$ are $C^1$ and time-like (since $q(\dot{\xi}_t )=1$). Moreover since Lorentz linear transforms act by isometry on $\Hyp^d$ and the $\Hyp^d$-Brownian motion $\dot{\xi}_t$ has isometries equivariant dynamics, it follows that Dudley's process has Lorentz equivariant dynamics. 

\paragraph{Asymptotic random variables.} Since the warping function $f:=\sinh$ satisfies the integrability condition 1-3 of Sect. \ref{intcond}, the angular process $\theta_t$ of $\dot{\xi}_t$ converges almost surely to a random variable $\theta_\infty \in \mathbb S^{d-1}$. There is another asymptotic random variable associated to $\xi_t$. We have indeed that $q( \xi_t, e_0 + \theta_\infty )$ converges almost-surely to some real random variable $R_\infty$. Geometrically this asymptotic random variable $R_\infty$ defines the position of some asymptotic affine hyperplan, whose direction is $q$-orthogonal to $e_0 + \theta_\infty$, see Fig. \ref{fig.hyperplan} below.

\begin{figure}[ht]
\begin{center}
\includegraphics[scale=0.7]{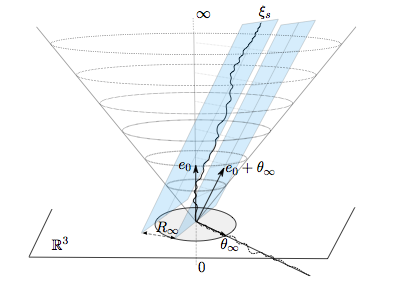}
\end{center}
\caption{Typical behavior of the relativistic diffusion in Minkowski space.\label{fig.hyperplan}}
\end{figure}
\noindent
We refer to \cite{ismael} for a detailed proof. Briefly, it follows from the decomposition
\begin{align*}
q( \xi_t, e_0 + \theta_\infty ) &= q( \xi_0, e_0 + \theta_\infty) + \int_0^t \left ( \cosh(r_s) - \sinh(r_s)\langle \theta_s, \theta_\infty \rangle \right )ds \\
&= q( \xi_0, e_0 + \theta_\infty) + \int_0^t \left ( e^{-r_s} - \sinh(r_s) d(\theta_s, \theta_\infty)^2  \right )ds,
\end{align*}
where $d(\cdot, \cdot)$ is the Riemannian distance on $\mathbb{S}^{d-1}$ and from the fact 
\[
\frac{r_t}{t} \underset{t \to \infty}{\longrightarrow} \frac{d-1}{2}\sigma^2, \qquad 
\limsup_{t \to + \infty} \log (d(\theta_t, \theta_\infty) ) \leq -\frac{d-1}{2}\sigma^2.
\]

\if{
We refer to \cite{ismael} for a detailed proof. Briefly, it follows from the decomposition
\begin{align*}
q( \xi_t, e_0 + \theta_\infty ) &= q( \xi_0, e_0 + \theta_\infty) + \int_0^t \left ( \cosh(r_s) - \sinh(r_s)\langle \theta_s, \theta_\infty \rangle \right )ds \\
&= q( \xi_0, e_0 + \theta_\infty) + \int_0^t \left ( e^{-r_s} - \sinh(r_s) d(\theta_s, \theta_\infty)^2  \right )ds,
\end{align*}
\begin{align*}
q( \xi_t, e_0 + \theta_\infty )&= q( \xi_0, e_0 + \theta_\infty) + \int_0^t \left ( e^{-r_s} - \sinh(r_s) d(\theta_s, \theta_\infty)^2  \right )ds,
\end{align*}
and from the fact that the last integral is convergent according to
\[
\frac{r_t}{t} \underset{t \to \infty}{\longrightarrow} \frac{d-1}{2}\sigma^2, \qquad 
\limsup_{t \to + \infty} \log (d(\theta_t, \theta_\infty) ) \leq -\frac{d-1}{2}\sigma^2,
\]
where $d(\cdot, \cdot)$ is the Riemannian distance on $\mathbb{S}^{d-1}$.
}\fi

In \cite{ismael}, Bailleul proved that the two variables $\theta_\infty$ and  $R_\infty$ are the only asymptotic variables associated to Dudley's diffusion. Namely, using coupling techniques, almost-coupling techniques, uniform continuity estimates for harmonic functions obtained via delicate Harnack inequalities,  he proved that
\begin{theo}[Bailleul]
The invariant $\sigma$-field of Dudley's diffusion $(\dot{\xi}_t, \xi_t)_{t \geq 0}$ is generated by the couple $(\theta_\infty, R_\infty) \in \mathbb S^{d-1}\times \mathbb R_+^*$.
\end{theo}

We propose here to use the devissage method to recover this result. For that we write the dynamics of Dudley diffusion in a new coordinate system $(\alpha_t, \beta_t, \gamma_t, h_t , \delta_t)$ which make appear a decomposition of the original diffusion into a sub-diffusion $(\alpha_t, \beta_t, \gamma_t)$ and a process $(h_t , \delta_t)$ with values in the group  $\R^{d-1} \times \R$ and which has equivariant dynamics. Then, we show that this $ \R^{d-1} \times \R$-valued process converges almost surely  to some asymptotic random variable $(h_\infty, \delta_\infty)$. We conclude the proof by checking that the sub-diffusion $(\alpha_t, \beta_t, \gamma_t)$ has a trivial Poisson boundary. In Remark \ref{rem.link} below, we explicit the link between $(h_\infty, \delta_\infty)$ and $(\theta_\infty, R_\infty)$, namely $h_\infty$ is a stereographical projection of $\theta_\infty$ and $\delta_\infty$ is a proportional to $R_\infty$. Our approach is inspired by Bailleul and Raugi's work \cite{Albert-Ismael} where the authors use Raugi's results on random walks on Lie groups to find the Poisson boundary of Dudley diffusion. 

\paragraph{New system of coordinates in $\Hyp^d \times \R^{1,d}$.} We find new coordinates in which Dudley's diffusion splits up in a sub-diffusion and a  process with values in the group $\R^{d-1} \times \R$. 
First, let us introduce Iwasawa coordinates on $\Hyp^{d}$ (see for instance \cite{fljlivre}) 
\[
\begin{matrix}
\R \times \R^{d-1} & \longrightarrow &  \Hyp^d \\
(\alpha, h )& \longmapsto & \left ( \begin{matrix} \frac{e^{\alpha}}{2} (1+ \vert h \vert^2 ) + \frac{e^{-\alpha}}{2}\\  \frac{e^{\alpha}}{2} (1-\vert h \vert^2 ) - \frac{e^{-\alpha}}{2} \\ e^{\alpha}h \end{matrix}  \right ) = \mathcal{T}_h \mathcal{D}_{\alpha}(e_0)
\end{matrix}
\] 
where $\mathcal{T}_{h}$ and $\mathcal{D}_{\alpha}$ are the following matrix of $q$-isometries of $\R{1,d}$.
\[
\mathcal{T}_{h} = \exp \left ( \begin{matrix} 0 & 0 & h^{t} \\ 0 &0 &-h^{t}  \\ h &h & 0   \end{matrix}  \right ) = \left ( \begin{matrix} 1+ \frac{\vert h \vert^2}{2} & \frac{\vert h \vert^2}{2} &h^{t} \\ -\frac{\vert h \vert^2}{2} & 1-\frac{\vert h \vert^2}{2} &-h^t \\ h &h& \mathrm{Id} \end{matrix}\right )
\]
and 
\[
\mathcal{D}_{\alpha}=\exp \left ( \begin{matrix} 0& \alpha & 0 \\ \alpha & 0 & 0  \\ 0& 0& 0   \end{matrix}  \right )= \left ( \begin{matrix} \cosh(\alpha)  & \sinh(\alpha) & 0 \\ \sinh(\alpha) &  \cosh(\alpha)  & 0 \\ 0 &0 & \mathrm{Id} \end{matrix} \right).
\]

Note that $\mathcal{T}_{h+h'}= \mathcal{T}_{h}\mathcal{T}_{h'}$ and $\mathcal{D}_{\alpha+ \alpha'}=\mathcal{D}_{\alpha}\mathcal{D}_{\alpha'}$ and that $(y:= e^{-\alpha}, h)$ are the classical half-space coordinates of hyperbolic space.
For $\xi \in \R^{1,d}$ denote by $(\xi)^{+}:= q( \xi, e_0 -e_1)$, $(\xi)^{-}:=q( \xi, e_0 + e_1)$ and $(\xi)^{\perp}:= (q( \xi, e_i ) )_{i=2,\dots,d} \in \R^{d-1} $.  We have the following $q$-orthogonal decomposition
\[\xi= (\xi)^{+} \frac{e_0 +e_1}{2} +(\xi)^{-}\frac{e_0 -e_1}{2}+\sum_{i=2}^d(\xi)_i^{\perp} e_i. \]
Now consider the new system of coordinates in $\Hyp^{d} \times \R^{1,d}$ given by the following diffeomorphism
\[
\begin{matrix}
\Hyp^{d} \times \R^{1,d}   & \longrightarrow & \R \times \R \times \R^{d-1} \times  \R^{d-1} \times \R \\
 (\mathcal{T}_{h}\mathcal{D}_{\alpha}(e_0) , \xi ) & \longmapsto & (\alpha, \beta, \gamma , h , \delta) := (\alpha, (\mathcal{T}_{h}^{-1} \xi )^{+}, (\mathcal{T}_{h}^{-1} \xi )^{\perp}, h, (\mathcal{T}_{h}^{-1} \xi )^{-} )  \end{matrix}.
\]

\begin{lemma}
In this new system of coordinates, Dudley's diffusion $(\alpha_t, \beta_t, \gamma_t, h_t , \delta_t)$ satisfies the following system of stochastic differential equations:
\begin{align*}
d \alpha_t &= \sigma d W_t + \frac{1}{2}\sigma^2(d-1) dt \\
d \beta_t &= e^{\alpha_t} dt \\
d \gamma_t &= \sigma e^{-\alpha_t}\beta_t dB_t \\
dh_t &= \sigma e^{-\alpha_t} dB_t \\
d \delta_t &= \left (e^{-\alpha_t} + \sigma^2 (d-1)\beta_t  e^{-2 \alpha_t} \right )dt + 2 \sigma e^{-\alpha_t} \gamma_t \cdot dB_t
\end{align*}
where $(W_t, B_t)$ is a standard Brownian motion of $\R \times \R^{d-1}$. 
\end{lemma}
\begin{proof}
Since in Iwasawa coordinates $(\alpha, h)$ the Laplacian in $\Hyp^d$ writes (see \cite{fljlivre} ) 
\begin{align} \label{Lap}
\Delta^{\Hyp} = e^{-2\alpha} \Delta^{\R^{d-1}}_{h} + \frac{\partial^2}{\partial \alpha^2} + (d-1) \frac{\partial}{\partial \alpha},
\end{align}
there exist $W_t$ and $B_t$ usual Brownian motion respectively of $\R$  and $\R^{d-1}$ such that
\begin{align*}
d \alpha_t &= \sigma dW_t + \frac{\sigma^2}{2}(d-1) dt, \\
d h_t &= \sigma e^{-\alpha_t} dB_t.
\end{align*}
Moreover, recall by definition $\dot{\xi}_t :=  \mathcal{T}_{h_t} \mathcal{D}_{\alpha_t}(e_0)$ thus
\begin{align*}
d \beta_t:= d (\mathcal{T}_{h_t}^{-1} \xi_t )^{+}&= d q( \mathcal{T}_{h_t}^{-1} \xi_t,e_0 -e_1) \\
&= d q( \xi_t, e_0 -e_1) \quad \mathrm{since} \ \mathcal{T}_{h_t}(e_0 -e_1)= e_0 -e_1 \\
&= q( \dot{\xi}_t, e_0-e_1) dt = e^{\alpha_t}dt.
\end{align*}
Moreover for $i=2\dots d$
\begin{align*} 
d(\gamma_t)^i := d q(\mathcal{T}_{h_t}^{-1}\xi_t, e_i )&= d q(\xi_t, \mathcal{T}_{h_t}(e_i) ) \\
&= q(\dot{\xi}_t,  \mathcal{T}_{h_t}(e_i) ) + q( \xi_t , \circ d\mathcal{T}_{h_t}(e_i) ) \\
&= q(\mathcal{D}_{\alpha_t},e_i) + q( \mathcal{T}_{h_t}^{-1} \xi_t, \mathcal{T}_{h_t}^{-1}\circ d\mathcal{T}_{h_t}(e_i) ) \\
&= 0 + q( \mathcal{T}_{h_t}^{-1} \xi_t, \circ dh_t^i (e_0 -e_1) ) \\
&= \beta_t \circ dh_t^i = \beta_t e^{-\alpha_t}dB_t^i.
\end{align*}
Finally, we have
\begin{align*}
d \delta_t &= d q( \mathcal{T}_{h_t}^{-1} \xi_t,e_0 + e_1) \\
&=d q( \xi_t, \mathcal{T}_{h_t}(e_0 +e_1) ) \\
&= q( \dot{\xi}_t, \mathcal{T}_{h_t}(e_0 +e_1) ) + q(\mathcal{T}_{h_t}^{-1} \xi_t, \mathcal{T}_{h_t}^{-1}\circ d\mathcal{T}_{h_t}(e_0 +e_1) ) \\
&= e^{-\alpha_t}dt + q( \mathcal{T}_{h_t}^{-1} \xi_t , 2\sum_{i=2}^d \circ d h_t^i e_i ) \\
&= e^{-\alpha_t}dt + 2\sum_{i=2}^d \gamma_t^i \circ dh_t^i \\
&= e^{-\alpha_t}dt +2\sigma e^{-\alpha_t}\gamma_t\cdot dB_t  + \sigma^2(d-1) \beta_t e^{-2 \alpha_t}dt.
\end{align*}
\end{proof}
\begin{remark}
We immediately see that $(\alpha_t, \beta_t, \gamma_t)_{t \geq 0}$ is a (d+1)-dimensional sub-diffusion of Dudley diffusion, and  the $\R^{d-1} \times \R$-valued process $(h_t, \delta_t)_{t \geq 0}$ have translation equivariant dynamics.
\end{remark}
\paragraph{Asymptotic behavior in the new coordinates.} \par
 We now establish the asymptotic behavior of Dudley's diffusion in coordinates $(\alpha, \beta, \gamma, h , \delta)$.  Namely, we show that $h_t$ and $\delta_t$ converge almost-surely to some asymptotic random variables $h_\infty \in \R^{d-1}$ and $\delta_\infty \in \mathbb R$. Moreover $\alpha_t$ and $\beta_t$ are transient and go almost-surely to infinity with $t$, and $\gamma_t$ is recurrent in $\mathbb R^{d-1}$. Note that  the asymptotic random variables  $(h_\infty, \delta_{\infty})$ are related to $(\theta_\infty,  R_\infty)$ in simple way which is explained in Remark \ref{rem.link} below. 
\begin{prop}
We have almost-surely 
\begin{align}
\frac{\alpha_t}{t} &\underset{t \to + \infty}{\longrightarrow} \frac{1}{2}\sigma^2(d-1) \label{alpha} \\
h_t &\underset{t \to + \infty}{\longrightarrow} h_\infty \\
\delta_t &\underset{t \to + \infty}{\longrightarrow} \delta_\infty 
\end{align}
where $h_\infty$ and $\delta_\infty$ are two asymptotic random variables. Moreover, we have almost-surely
\begin{align}
\int_0^{+\infty} \beta^2_u e^{-2 \alpha_u} du =+ \infty. \label{infini}
\end{align}
\end{prop}

\begin{figure}[ht]
\begin{center}
\includegraphics[scale=0.7]{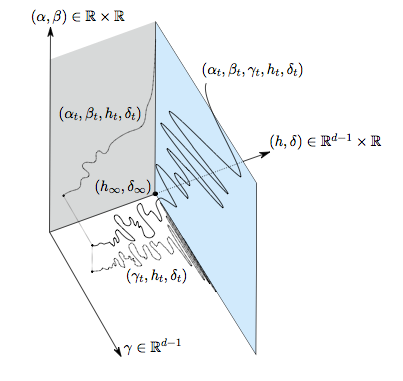}
\end{center}
\caption{Asymptotics of Dudley's diffusion in the new coordinates.\label{fig.newcoord}}
\end{figure}

\begin{proof}
Since $\alpha_t = \alpha_0 + \frac{1}{2}\sigma^2(d-1) t + \sigma W_t$, the convergence \eqref{alpha} follows from the law of iterated logarithm. Hence, the integrants in the expression defining $h_t$ and $\delta_t$ are dominated by $e^{-(\frac{1}{2}\sigma^2(d-1) -\eps)t}$ for some $\eps >0$ fixed and for $t$ sufficiently large, so that $h_t$ and $\delta_t$ converge almost surely.
Let us now check \eqref{infini} and denote by $u_t:= \beta_t e^{-\alpha_t}$. Then $u_t >0 $ is solution of 
\[
du_t = u_t \sigma dW_t +\left (1 - \frac{d\sigma^2}{2} u_t\right) dt
\] 
and since for $u \geq 0$ $\frac{1}{2} - \frac{d^2 \sigma^4}{8} u^2 \leq 1- \frac{d\sigma^2}{2} u $ we obtain by comparison theorem that almost surely $u_t \geq v_t$ where $v_t \geq 0$ is started at $u_0$ and  solution of
\[
dv_t = v_t \sigma dW_t + \left(\frac{1}{2} - \frac{d^2\sigma^4}{8} (v_t)^2\right) dt.
\]
Suppose by contradiction that $\int_0^t (u_s)^2ds$ is finite. Then, since $v_t \leq u_t$, so does $\int_0^t (v_s)^2ds$ but
\[ v_t = v_0 + \sigma \int_0^t v_s dW_s + \frac{1}{2}t - \frac{d^2\sigma^4}{8} \int_0^t (v_s)^2 ds \]
and thus $v_t - \frac{1}{2}t$ would converge almost surely. This contradicts the integrability of $(v_t)^2$.\end{proof}

\begin{remark}\label{rem.link}
The asymptotic random variable $\theta_\infty$ and $h_\infty$ define the same asymptotic line in the light cone. Namely we have
\[
\frac{1}{q(e_0, \mathcal{T}_{h_\infty} (e_0 + e_1))}\mathcal{T}_{h_\infty} (e_0 + e_1) = e_0 + \theta_\infty,
\]
or more explicitly, $h_\infty$ is a stereographical projection of $\theta_\infty$
\[
\theta_\infty = \frac{1}{1 + \vert h_\infty \vert^2} \left ( (1- \vert h_\infty \vert^2)e_1 + 2 \sum_{i=2}^{d} h_\infty^i e_i \right ).
\]
Moreover $R_\infty$ and $\delta_\infty$ are proportional 
\[
R_\infty = \frac{1}{1 + \vert h_\infty \vert^2} \delta_\infty.
\]
\end{remark}


\paragraph{Poisson boundary.} To finally recover Bailleul's result using the devissage method, we have to show that the sub-diffusion $(\alpha_t, \beta_t, \gamma_t)$ has a trivial Poisson boundary. 
\begin{prop}
The sub-diffusion $(\alpha_t, \beta_t, \gamma_t)$ has a trivial Poisson boundary.
\end{prop}
\begin{proof}
We aim to find a shift-coupling for two copies of the process $(\alpha_t, \beta_t, \gamma_t)$. Fix $(\alpha,\beta, \gamma)$ and $(\bar{\alpha}, \bar{\beta}, \bar{\gamma})$ in $\R \times \R \times \R^{d-1}$. Let $W_t$ and $\bar{W}_t$ be two independent Brownian motion on $\R$ starting at $0$ and set for $t\geq 0$
\[
\alpha_t := \alpha + \sigma W_t + \frac{d-1}{2}\sigma^2 t , \qquad
\tilde{\alpha}_t := \bar{\alpha} +  \sigma \tilde{W}_t +\frac{d-1}{2}\sigma^2 t. 
\]
Then $\alpha_t -\tilde{\alpha}_t$ is a Brownian motion starting at $\alpha -\tilde{\alpha}$ and $\tilde{T}:= \inf \{ t \geq0; \   \alpha_t = \tilde{\alpha}_t\}$ is finite almost-surely. Next, define
\[
\bar{\alpha}_t := \tilde{\alpha}_t \ind_{t \leq \tilde{T} } + \alpha_{t} \ind_{t > \tilde{T}}, 
\]
and set for $t\geq 0$
\[
\beta_t := \beta + \int_0^t e^{\alpha_s}ds, \qquad \bar{\beta}_t := \bar{\beta} + \int_0^t e^{\bar{\alpha}_s}ds,
\]
and 
\[
S := \inf \left \{  t \geq \tilde{T} , \beta_t = \max \left ( \beta_{\tilde{T}}, \bar{\beta}_{\tilde{T}}  \right ) \right \} , \qquad
\bar{S} := \inf \left \{  t \geq \tilde{T} , \bar{\beta}_t = \max \left ( \beta_{\tilde{T}}, \bar{\beta}_{\tilde{T}}  \right )
\right \}.
\]

\if{
\begin{align*}
\alpha_t &:= \alpha + \sigma W_t + \frac{d-1}{2}\sigma^2 t \\
\tilde{\alpha}_t &:= \bar{\alpha} +  \sigma \tilde{W}_t +\frac{d-1}{2}\sigma^2 t. 
\end{align*}
Then $\alpha_t -\tilde{\alpha}_t$ is a brownian motion starting at $\alpha -\tilde{\alpha}$ and $\tilde{T}:= \inf \{ t \geq0; \   \alpha_t = \tilde{\alpha}_t\}$ is finite almost surely. Then define
\[
\bar{\alpha}_t := \tilde{\alpha}_t \ind_{t \leq \tilde{T} } + \alpha_{t} \ind_{t > \tilde{T}}. 
\]
Then set for $t\geq 0$
\begin{align*}
\beta_t &:= \beta + \int_0^t e^{\alpha_s}ds \\
\bar{\beta}_t &:= \bar{\beta} + \int_0^t e^{\bar{\alpha}_s}ds
\end{align*}
and 
\begin{align*}
S &:= \inf \left \{  t \geq \tilde{T} , \beta_t = \max \left ( \beta_{\tilde{T}}, \bar{\beta}_{\tilde{T}}  \right ) \right \} \\
\bar{S} &:= \inf \left \{  t \geq \tilde{T} , \bar{\beta}_t = \max \left ( \beta_{\tilde{T}}, \bar{\beta}_{\tilde{T}}  \right )
\right \}.
\end{align*}
}
\fi

Since $\beta_t$ and $\bar{\beta}_t$ are strictly increasing $S$ and $\bar{S}$ are almost surely finite and for $s\geq 0$ we obtain $\beta_{S+s}= \bar{\beta}_{\bar{S} + s}$ (and since $S\geq \tilde{T}$ we have also $\alpha_{S+s} = \bar{\alpha}_{S+s}$). 
Now, consider $t\mapsto B(t)$ and $t \mapsto \hat{B}(t)$ two independent Brownian motions in $\R^{d-1}$ starting at $0$  and set 
\[
\gamma_t := B\left ( \int_S^{t \vee S} \left (e^{-\alpha_s}\beta_s \right)^2 ds  \right ) + \hat{B} \left ( \int_0^{t\wedge S} \left (e^{-\alpha_s}\beta_s \right)^2 ds  \right ) + \gamma.
\]
Let us denote by $t \mapsto \bar{B}(t)$ the Brownian motion starting at $0$ and being the image of $B$ by the linear reflection in the hyperplan orthogonal to the vector 
\[
\hat{B}\left ( \int_0^{ S} \left (e^{-\alpha_s}\beta_s \right)^2 ds  \right ) - \hat{B}\left ( \int_0^{ \bar{S} } \left (e^{-\bar{\alpha}_s} \bar{\beta}_s \right)^2 ds  \right ) +\gamma - \bar{\gamma},
\] 
and set
\[
\bar{\gamma}_t:= \bar{B}\left ( \int_{\bar{S}}^{t \vee \bar{S}} \left (e^{-\bar{\alpha}_s}\bar{\beta}_s \right)^2 ds  \right ) + \hat{B} \left ( \int_0^{t\wedge \bar{S}} \left (e^{-\bar{\alpha}_s}\bar{\beta}_s \right)^2 ds  \right ) + \bar{\gamma}.
\]
Since a Brownian motion in $\R^{d-1}$ reaches in a finite time every affine hyperplan, we obtain that
\[
R:= \inf \left \{ t \geq 0, B(t) + \hat{B}\left ( \int_0^{ S} \left (e^{-\alpha_s}\beta_s \right)^2 ds  \right )  + \gamma = \bar{B}(t) + \hat{B}\left ( \int_0^{ \bar{S} } \left (e^{-\bar{\alpha}_s} \bar{\beta}_s \right)^2 ds  \right ) + \bar{\gamma}  \right \},
\]
is almost-surely finite. Finally, introducing $T$ and $\bar{T}$ such that 
\[
\int_S^{T} \left (e^{-\alpha_s}\beta_s \right)^2 ds  = R, \qquad \int_{\bar{S}}^{\bar{T}} \left (e^{-\bar{\alpha}_s}\bar{\beta}_s \right)^2 ds  = R,
\]
 we have $\gamma_T = \bar{\gamma}_{\bar{T}}$, hence the result.
\end{proof}
We are finally in position to apply the d\'evissage scheme to the diffusion $(x_t, g_t)_{t \geq 0}$ where
$x_t:=(\alpha_t,\beta_t,\gamma_t) \in X:=\R \times \R \times \R^{d-1}$ and $g_t:= (h_t, \delta_t) \in G:=\R^{d-1} \times \R$. Namely, from Theorem \ref{theo.trivial}, we can conclude that

\begin{theo}
For all $(\dot{\xi}, \xi) \in \Hyp^d \times \R^{1,d}$, the invariant $\sigma$-algebra of Dudley's diffusion starting from $(\dot{\xi}, \xi)$ coincides with $\sigma(h_\infty, \delta_\infty)$ up to $\Prob_{(\dot{\xi},\xi)}$ negligible sets.
\end{theo}
Note that by Remark \ref{rem.link},  we have $\sigma(h_\infty, \delta_\infty) = \sigma( \theta_\infty, R_\infty)$ i.e. we recover precisely Bailleul's result.

\bibliographystyle{alpha}

\end{document}